\documentclass[12pt]{amsart}

\usepackage{amsfonts, amsthm, amsmath}

\usepackage{rotating}

\usepackage{tikz}

\usepackage{graphics}

\usepackage{amssymb}

\usepackage{amscd}

\usepackage[latin2]{inputenc}

\usepackage{t1enc}

\usepackage[mathscr]{eucal}

\usepackage{indentfirst}

\usepackage{graphicx}

\usepackage{graphics}

\usepackage{pict2e}

\usepackage{mathrsfs}

\usepackage{enumerate}
\usepackage[pagebackref]{hyperref}
\hypersetup{backref, pagebackref, colorlinks=true}
\usepackage{cite}
\usepackage{color}
\usepackage{epic}
\usepackage{hyperref} 
\numberwithin{equation}{section}
\theoremstyle{plain}

\newtheorem{theorem}{Theorem}[section]

\theoremstyle{definition}

\newtheorem{Def}[theorem]{Definition}

\newtheorem{example}[theorem]{Example}

\newtheorem{remark}[theorem]{Remark}

\newtheorem{?}[theorem]{Problem}

\def\boxit#1{\leavevmode\hbox{\vrule\vtop{\vbox{\kern.33333pt\hrule
    \kern1pt\hbox{\kern1pt\vbox{#1}\kern1pt}}\kern1pt\hrule}\vrule}}

\usepackage{collectbox}



\begin{document}

\title[Multiranks and classical theta functions]{Multiranks and classical theta functions}

\author[S. Fu]{Shishuo Fu}
\address[Shishuo Fu]{College of Mathematics and Statistics, Chongqing University, Huxi Campus LD506, Chongqing 401331, P.R. China}
\email{fsshuo@cqu.edu.cn}

\author[D. Tang]{Dazhao Tang}

\address[Dazhao Tang]{College of Mathematics and Statistics, Chongqing University, Huxi Campus LD208, Chongqing 401331, P.R. China}
\email{dazhaotang@sina.com}

\date{\today}

\begin{abstract}
Multiranks and new rank/crank analogs for a variety of partitions are given, so as to imply combinatorially some arithmetic properties enjoyed by these types of partitions. Our methods are elementary relying entirely on the three classical theta functions, and are motivated by the seminal work of Ramanujan, Garvan, Hammond and Lewis.
\end{abstract}



\maketitle


\section{Introduction}\label{sec:intro}

Ramanujan's celebrated partition congruences have been a continuing source of inspiration and have motivated a tremendous amount of research for almost a century. Let $p(n)$ denote the number of partitions of $n$, then the aforementioned congruences (each being the first in an infinite family of congruences) read
\begin{align}
\label{p5}p(5n+4)&\equiv 0 \pmod 5,\\
\label{p7}p(7n+5)&\equiv 0 \pmod 7,\\
\label{p11}p(11n+6)&\equiv 0 \pmod {11}.
\end{align}

See Ramanujan \cite{Ram} and Winquist \cite{Win} for elementary proofs of \eqref{p5}--\eqref{p7} and \eqref{p11} respectively. Dyson \cite{Dys} was the first to consider these congruences combinatorially. In 1944, he discovered empirically the following remarkable results:

\begin{align*}
N(0,5,5n+4)=N(1,5,5n+4)=\cdots=N(4,5,5n+4)=\dfrac{p(5n+4)}{5},\\
N(0,7,7n+5)=N(1,7,7n+5)=\cdots=N(6,7,7n+5)=\dfrac{p(7n+5)}{7},
\end{align*}
where $N(m,t,n)$ is the number of partitions of $n$ with \emph{rank} congruent to $m$ modulo $t$. Recall that the rank of a partition, as defined by Dyson, is the largest part minus the number of parts. This partition statistic evidently provides a combinatorial explanation of \eqref{p5} and \eqref{p7}, but unfortunately fails to do so for \eqref{p11}. Dyson \cite{Dys} then conjectured the existence of certain partition statistic that will serve the same purpose as the rank for Ramanujan's third congruence \eqref{p11}, and he named it the \emph{``crank''}. In 1988, upon applying two identities from Ramanujan's ``lost'' notebook (see Andrews \cite{And} for an introduction), Garvan \cite{Gar} discovered the vector crank, which eventually led to the discovery of the elusive crank by Andrews and Garvan \cite{AG1}. The first of Ramanujan's two identities concerns
\begin{align*}
F_5(q):=\dfrac{(q;q)_{\infty}}{(\zeta_5q;q)_{\infty}(\zeta_5^{-1}q;q)_{\infty}},
\end{align*}
where $\zeta_5$ is a primitive fifth root of unity. Here and in the sequel, we will use the following customary $q$-series notations:
\begin{align*}
(a;q)_{\infty} &=\prod_{k=0}^{\infty}(1-aq^{k}),\quad |q|<1,\\
(a_{1},a_{2},\cdots,a_{n};q)_{\infty}&=(a_{1};q)_{\infty}(a_{2};q)_{\infty}\cdots(a_{n};q)_{\infty}.
\end{align*}
It is precisely this $F_5(q)$ that motivated Garvan to consider the following generating function:
\begin{align*}
\sum_{n=0}^{\infty}\sum_{m=-\infty}^{\infty}N_{V}(m,n)z^{m}q^{n}=\dfrac{(q;q)_{\infty}}{(zq;q)_{\infty}(z^{-1}q;q)_{\infty}}.
\end{align*}
This amounts to give unified explanations for \eqref{p5}--\eqref{p11}:
\begin{align}
\label{NV5}N_V(0,5,5n+4)=N_V(1,5,5n+4)=\cdots=N_V(4,5,5n+4) &=\dfrac{p(5n+4)}{5},\\
\label{NV7}N_V(0,7,7n+5)=N_V(1,7,7n+5)=\cdots=N_V(6,7,7n+5) &=\dfrac{p(7n+5)}{7},\\
\label{NV11}N_V(0,11,11n+6)=N_V(1,11,11n+6)=\cdots=N_V(10,11,11n+6) &=\dfrac{p(11n+6)}{11}.
\end{align}
For the definition of $N_V(m,n)$ and further details, the readers are referred to Garvan's original paper \cite{Gar}. One of the key components in Garvan's derivation of \eqref{NV5}--\eqref{NV11} is the famous Jacobi's triple product identity \cite[p.21, Theorem~2.8]{Andr1}. Using Ramanujan's notation \cite[p.6, Definition 1.2.1]{Ber}, it is given, for $|ab|<1$, by
\begin{align}
\label{JTP}f(a,b):=\sum_{n=-\infty}^{\infty}a^{n(n+1)/2}b^{n(n-1)/2}=(-a,-b,ab;ab)_{\infty}.
\end{align}
Making appropriate substitutions for $a$ and $b$ in \eqref{JTP} will then produce the product representations of the classical theta functions studied by Gauss, Jacobi and Ramanujan:
\begin{align}
\label{f}f(-q)&:=f(-q,-q^2)=\sum_{n=-\infty}^{\infty}(-1)^nq^{n(3n-1)/2}=(q;q)_{\infty},\\
\label{phi}\varphi(-q)&:=f(-q,-q)=\sum_{n=-\infty}^{\infty}(-1)^nq^{n^2}=\dfrac{(q;q)_{\infty}}{(-q;q)_{\infty}},\\
\label{psi}\psi(-q)&:=f(-q,-q^3)=\sum_{n=0}^{\infty}(-q)^{n(n+1)/2}=\dfrac{(q^2;q^2)_{\infty}}{(-q;q^2)_{\infty}}.
\end{align}

In 2004, Hammond and Lewis \cite{HL} defined the birank for $2$-colored partitions and explained that the residue of the birank modulo $5$ divided 2-colored partitions of $n$ into $5$ equinumerous classes provided $n\equiv2, 3~\textrm{or}~4\pmod{5}$. Garvan \cite{Gar2} then generalized their idea to define a multirank for multi-colored partitions, so as to give similar combinatorial explanations for two families of congruences (see Theorem~\ref{Gmulti} and Theorem~\ref{Gmultir} below) enjoyed by these multipartitions.

Notice that $f(-q)$'s reciprocal, namely $1/(q;q)_{\infty}$, is exactly the generating function of ordinary partitions, while the reciprocal of $\varphi(-q)$ (resp. $\psi(-q)$) is the generating function of overpartitions (resp. partitions with odd parts distinct, abbreviated as ``pod'' in what follows).  With this observation in mind, one naturally wonders if analogous generalizations can be obtained for multi-overpartitions and multi-pods. We give an affirmative answer in this paper and use the same idea to tackle various types of partition congruences in the literature that are missing the rank/crank explanations.

The rest of the paper is organized as follows. In Section \ref{sec:over and pod}, we unify the notion of multirank for multipartitions, multi-overpartitions and multi-pods, each of which implies a Ramanujan type congruence for the corresponding partitions (see Theorem~\ref{thm:multi-over} and Theorem~\ref{thm:multi-pod pa}). Next in Section \ref{sec: cubic}, we utilize a similar idea to generalize Reti's \cite{Ret} crank analog for cubic partitions to overcubic partitions (see Theorem~\ref{thm:overcubic}), as well as their closely related variations (see Theorem~\ref{thm:4c}). We conclude in the last section with some remarks and further problems.

\section{A unified multirank for three kinds of multipartitions}\label{sec:over and pod}
\subsection{Garvan's multirank for multipartitions}
Let $\mathcal{P}$ (resp. $\mathcal{\overline{P}}$, $\mathcal{\hat{P}}$) denote the set of all partitions (resp. overpartitions, pods). A multi-partition (resp. multi-overpartition, multi-pod) with $t$ components or a $t$-colored partition (resp. $t$-colored overpartition, $t$-colored pod) of $n$ is simply a $t$-tuple
\begin{align*}
\overrightarrow{\pi}=(\pi_{1},\pi_{2},\cdots,\pi_{t})\in\mathcal{P}\times\mathcal{P}\cdots\times\mathcal{P}
=\mathcal{P}^{t} (\textrm{resp}.~\mathcal{\overline{P}}^{t}, \mathcal{\hat{P}}^t),
\end{align*}
where
\begin{align*}
\sum_{i=1}^{t}|\pi_{i}|=n.
\end{align*}

We use $p_t(n)$ (resp. $\overline{p}_t(n)$, $\hat{p}_t(n)$) to denote the number of $t$-colored partitions (resp. $t$-colored overpartitions, $t$-colored pods) of $n$. And for a given partition $\pi$ of any type, we use $|\pi|$ and $\ell(\pi)$ to denote the sum of parts and the number of parts, respectively.

For completeness and comparison, let us briefly recall here Garvan's multirank generalization (see \cite{Gar2}) of the Hammond-Lewis birank.

\begin{theorem}[Theorem 7.1, \cite{Gar2}]\label{Gmulti}
Let $t>3$ be prime.
\begin{itemize}
\item[(i)] If $24n+1$ is a quadratic nonresidue mod t, then
\begin{align*}
p_{t-1}(n)\equiv 0 \pmod{t}.
\end{align*}
\item[(ii)] If $8n+1$ is a quadratic nonresidue mod t, then
\begin{align*}
p_{t-3}(n)\equiv 0 \pmod{t}.
\end{align*}
\end{itemize}
\end{theorem}
\begin{remark}
We note that the second family (ii) above also appeared in Andrews' survey paper \cite{Andr2}, which discussed other interesting aspects of multipartitions.
\end{remark}

For $s$ even and $\overrightarrow{\pi}=(\pi_{1},\pi_{2},\cdots,\pi_{s})\in\mathcal{P}^{s}$, Garvan defined the generalized Hammond-Lewis multirank by
\begin{align*}
\textrm{gHL-multirank}(\overrightarrow{\pi})=\sum_{k=1}^{s/2}k(\ell(\pi_{k})-\ell(\pi_{s+1-k})).
\end{align*}

The ``rank-ish'' of this new statistic is justified by the following refinement.
\begin{theorem}[Theorem 7.2, \cite{Gar2}]\label{Gmultir}
Let $t>3$ be prime.
\begin{itemize}
\item[(i)] The residue of the generalized-Hammond-Lewis-multirank mod $t$ divides the multipartitions of $n$ with $s=t-1$ components into $t$ equal classes provided $24n+1$ is a quadratic nonresidue mod $t$.
\item[(ii)] The residue of the generalized-Hammond-Lewis-multirank mod $t$ divides the multipartitions of $n$ with $s=t-3$ components into $t$ equal classes provided
$8n+1$ is a quadratic nonresidue mod $t$.
\end{itemize}
\end{theorem}
Recall that an \emph{overpartition}, as defined by Corteel and Lovejoy \cite{CLo}, is a partition in which the first occurrence of each distinct number may be overlined. The study of the \emph{pod} function was initiated by Alladi \cite{All,All2} combinatorially as well as by Hirschhorn and Sellers \cite{HS} arithmetically.
Motivated by the above results, we explore the power of the generalized-Hammond-Lewis-multirank further with multi-overpartition and multi-pods. It is understood that for $\pi\in\mathcal{\overline{P}}$ (resp. $\pi\in\mathcal{\hat{P}}$), $\ell(\pi)$ counts all parts in $\pi$, both overlined and non-overlined (resp. both even and odd), and we stick to the following notation for brevity.
\begin{Def}
For $t$ even, we define the multirank for all three kinds of multipartitions to be
\begin{align}\label{multirank}
\overline{r}(\overrightarrow{\pi})=\sum_{k=1}^{t/2}k(\ell(\pi_{k})-\ell(\pi_{t+1-k})),\text{
 for }
\overrightarrow{\pi}=(\pi_{1},\pi_{2},\cdots,\pi_{t})\in\mathcal{P}^t, \mathcal{\overline{P}}^t \text{ or }\mathcal{\hat{P}}^t.
\end{align}
\end{Def}

In the following two subsections, we will present congruences for $\overline{p}_{t}(n)$ and $\hat{p}_t(n)$ that are comparable to Theorem~\ref{Gmulti}~(i), as well as combinatorial explanations in terms of our multirank $\overline{r}$, a situation that is similar to Theorem~\ref{Gmultir}. Let $\overline{N}_s(m,n)$ (resp. $\hat{N}_s(m,n)$) denote the number of $s$-colored overpartitions (resp. pods) of $n$ with multirank $m$, and let $\overline{N}_s(m,t,n)$ (resp. $\hat{N}_s(m,t,n)$) denote the number of $s$-colored overpartitions (resp. pods) of $n$ with multirank congruent to $m$ modulo $t$. In general, when $*(m,n)$ enumerates the number of certain partitions (or multipartitions) of $n$ with certain statistic being $m$, then $*(m,t,n)$ counts those with that statistic  congruent to $m$ modulo $t$.

\subsection{Multirank for multi-overpartitions}
Now we are ready to state the parallel results for multi-overpartitions first.
\begin{theorem}\label{thm:multi-over}
Let $t$ be any odd prime, if $n\geq 1$ is any quadratic nonresidue modulo $t$, then for all $0\leq i\leq t-1$, we have
\begin{align}\label{rank:multi-over}
\overline{N}_{t-1}(i,t,n)=\dfrac{\overline{p}_{t-1}(n)}{t}\equiv 0 \pmod 2.
\end{align}
In particular, we have
\begin{align}\label{cong:multi-over}
\overline{p}_{t-1}(n)\equiv0\pmod{2t}.
\end{align}
\end{theorem}
\begin{proof}
By Definition~\eqref{multirank}, it is straightforward to write the generating function for $\overline{N}_{t-1}(m,n)$:
\begin{align}\label{gf:over-part}
\sum_{m=-\infty}^{\infty}\sum_{n=0}^{\infty}\overline{N}_{t-1}(m,n)z^{m}q^{n}=
\dfrac{(-zq,-z^{-1}q,\cdots,-z^{s}q,-z^{-s}q;q)_{\infty}}{(zq,z^{-1}q,\cdots,z^{s}q,z^{-s}q;q)_{\infty}},
\end{align}
where $t=2s+1$. By setting $z=\zeta_{t}=e^{2\pi i/t}$ in \eqref{gf:over-part}, we see that
\begin{align}\label{gf2:over-part}
 &\sum_{m=-\infty}^{\infty}\sum_{n=0}^{\infty}\overline{N}_{t-1}(m,n)\zeta_{t}^{m}q^{n}=\sum_{n=0}^{\infty}\sum_{i=0}^{t-1}\overline{N}_{t-1}(i,t,n)\zeta_{t}^{i}q^{n}\\
 =&\dfrac{(-\zeta q,-\zeta^{-1}q,\cdots,-\zeta^{s}q,-\zeta^{-s}q;q)_{\infty}}{(\zeta q,\zeta^{-1}q,\cdots,\zeta^{s}q,\zeta^{-s}q;q)_{\infty}}\nonumber\\
 =&\dfrac{(-q^{t};q^{t})_{\infty}(q;q)_{\infty}}{(q^{t};q^{t})_{\infty}(-q;q)_{\infty}}\nonumber\\
 =&\dfrac{(-q^{t};q^{t})_{\infty}\sum_{n=-\infty}^{\infty}(-1)^{n}q^{n^{2}}}{(q^{t};q^{t})_{\infty}}.\nonumber
\end{align}
Here the third equality uses that $\zeta_{t}$ is a primitive $t$-th root of unity, hence satisfies the following identity
\begin{align*}
(1-x^t)=(1-x)(1-\zeta_tx)\cdots(1-\zeta_t^{t-1}x).
\end{align*}
The last equality relies on $\varphi(-q)$'s product representation \eqref{phi}. Note that if $n$ is a quadratic nonresidue mod $t$, then the coefficient of $q^{n}$ in \eqref{gf2:over-part} is always zero, which gives
\begin{align}\label{poly}
\sum_{i=0}^{t-1}\overline{N}_{t-1}(i,t,n)\zeta_{t}^{i}=0.
\end{align}
We note that the left-hand side of \eqref{poly} is a polynomial in $\zeta_{t}$ over $\mathbb{Z}$. It follows that
\begin{align*}
\overline{N}_{t-1}(i,t,n)=\overline{N}_{t-1}(j,t,n)
\end{align*}
for all $0\leq i\leq j\leq t-1$, since $t$ is prime and the minimal polynomial for $\zeta_{t}$  over $\mathbb{Q}$ is
\begin{align*}
1+x+x^{2}+\cdots+x^{t-1}.
\end{align*}
Moreover, to see why each $\overline{N}_{t-1}(i,t,n)$ must be even, one simply recalls the natural ``toggling involution'' defined on overpartitions. Namely, for any given overpartition $\overline{\lambda}$, we examine its largest part $\lambda_1$. If $\lambda_1$ is already overlined, then we remove the bar to make it non-overlined, and if $\lambda_1$ is non-overlined, then we overline it. This is clearly a well-defined involution on the set of all non-empty overpartitions, and it can be generalized to apply on multi-overpartitions, as long as we predetermine the order of parts with equal size but different color, so that we can always find ``the'' largest part. Then the key observation to make here is that our multirank $\overline{r}$ is preserved under this involution. Therefore the two multi-overpartitions paired up by this involution are counted by the same $\overline{N}_{t-1}(i,t,n)$, forcing it to be an even number. This establishes \eqref{rank:multi-over}, which in turn implies \eqref{cong:multi-over}.
\end{proof}
\begin{remark}
Several remarks on Theorem~\ref{thm:multi-over} are in order. First note that \eqref{cong:multi-over} was first obtained by Keister, Sellers and Vary \cite{KSV}, but the complete multirank explanation in \eqref{rank:multi-over} appears to be new. Secondly, Bringmann and Lovejoy \cite{BL} were the first to study arithmetic properties for overpartition pairs and derived the $t=3$ case of \eqref{cong:multi-over}. They also defined certain rank for overpartition pairs that would yield a combinatorial explanation similar to \eqref{rank:multi-over}. Their rank is different from our multirank and seems to be more involved. Moreover, Chen and Lin defined in \cite{CLi} three different kinds of rank to give combinatorial explanations of $\overline{p}_{2}(3n+2)\equiv0\pmod{3}$. The first of their three ranks is exactly our multirank for $t=3$, so from this perspective, Theorem~\ref{thm:multi-over} can be viewed as a natural generalization of the work of Chen and Lin. Lastly, there is another type of ``$n$-color overpartition'' studied by Lovejoy and Mallet \cite{LM}, which provides combinatorial interpretations to several interesting $q$-series identities involving twisted divisor functions.
\end{remark}

\subsection{Multirank for multi-pods}
We continue to develop analogous congruence results for $\hat{p}_t(n)$, using the same multirank $\overline{r}$ as defined in \eqref{multirank}.
\begin{theorem}\label{thm:multi-pod pa}
Let $t$ be any odd prime, if for $n\geq 0$, $8n+1$ is a quadratic nonresidue modulo t, then for all $0\leq i\leq t-1$, we have
\begin{align}\label{rank:multi-pod}
\hat{N}_{t-1}(i,t,n)=\dfrac{\hat{p}_{t-1}(n)}{t}.
\end{align}
In particular, we have
\begin{align}\label{cong:multi-pod}
\hat{p}_{t-1}(n)\equiv 0 \pmod t.
\end{align}
\end{theorem}
\begin{proof}
Similarly, we begin with the generating function for $\hat{N}_{t-1}(m,n)$ and assume $t=2s+1$:
\begin{align}\label{gf:multi-pod}
\sum_{m=-\infty}^{\infty}\sum_{n=0}^{\infty}\hat{N}_{t-1}(m,n)z^{m}q^{n}=
\dfrac{(-zq,-z^{-1}q,\cdots,-z^{s}q,-z^{-s}q;q^{2})_{\infty}}{(zq^{2},z^{-1}q^{2},\cdots,z^{s}q^{2},z^{-s}q^{2};q^{2})_{\infty}},
\end{align}
By putting $z=\zeta_{t}=e^{2\pi i/t}$ again in \eqref{gf:multi-pod}, we see that
\begin{align*}
 &\sum_{m=-\infty}^{\infty}\sum_{n=0}^{\infty}\hat{N}_{t-1}(m,n)\zeta_{t}^{m}q^{n}=\sum_{n=0}^{\infty}\sum_{i=0}^{t-1}\hat{N}_{t-1}(i,t,n)\zeta_{t}^{i}q^{n}\\
 =&\dfrac{(-\zeta q,-\zeta^{-1}q,\cdots,-\zeta^{s}q,-\zeta^{-s}q;q^{2})_{\infty}}{(\zeta q^2,\zeta^{-1}q^{2},\cdots,\zeta^{s}q^{2},\zeta^{-s}q^{2};q^{2})_{\infty}}\nonumber\\
 =&\dfrac{(-q^t;q^{2t})_{\infty}(q^{2};q^{2})_{\infty}}{(q^{2t};q^{2t})_{\infty}(-q;q^{2})_{\infty}}\nonumber\\
 =&\dfrac{(-q^{t};q^{2t})_{\infty}\sum_{n=0}^{\infty}(-q)^{n(n+1)/2}}{(q^{2t};q^{2t})_{\infty}}.
\end{align*}
Here the third equality is again due to $\zeta_{t}$ being a primitive $t$-th root of unity, and the last equality relies on $\psi(-q)$'s product representation \eqref{psi}. Upon completing the square, we arrive at
\begin{align}\label{gf2:over-pod}
\sum_{n=0}^{\infty}\sum_{i=0}^{t-1}\hat{N}_{t-1}(i,t,n)\zeta_{t}^{i}q^{8n+1}=\dfrac{(-q^{8t};q^{16t})_{\infty}\sum_{n=0}^{\infty}(-1)^{n(n+1)/2}q^{(2n+1)^2}}{(q^{16t};q^{16t})_{\infty}}.
\end{align}
We see that in the $q$-expansion on the right side of \eqref{gf2:over-pod} the coefficient of $q^{n}$ is zero when $n$ is a quadratic nonresidue modulo $t$. This together with a similar argument using the minimal polynomial for $\zeta_t$ gives \eqref{rank:multi-pod}, which in turn implies \eqref{cong:multi-pod} and completes the proof.
\end{proof}

\begin{table}[htbp]\caption{Three kinds of rank for pod pairs of $5$}
\centering
\begin{tabular}{|c||c|c|c|}
pod pairs  &Chan-Mao rank &Chan-Mao crank & Multirank $\overline{r}$ \\
$(5,0)$ &2 &0 &1  \\
$(4~1,0)$ &0 &1 &2  \\
$(3~2,0)$ &0 &1 &2  \\
$(2~2~1,0)$ &$-2$ &2 &3  \\
$(4,1)$ &0 &1 &0  \\
$(3~1,1)$ &$-1$ &0 &1 \\
$(2~2,1)$ &$-2$ &2 &1  \\
$(3,2)$ &1 &$-1$ &0  \\
$(2~1,2)$ &$-1$ &0 &1  \\
$(2,2~1)$ &$-1$ &0 &$-1$  \\
$(2,3)$ &0 &1 &0  \\
$(1,2~2)$ &0 &$-2$ &$-1$  \\
$(1,3~1)$ &$-1$ &0 &$-1$  \\
$(1,4)$ &1 &$-1$ &0  \\
$(0,2~2~1)$ &0 &$-2$ &$-3$  \\
$(0,3~2)$ &1 &$-1$ &$-2$  \\
$(0,4~1)$ &1 &$-1$ &$-2$  \\
$(0,5)$ &2 &0  &$-1$  \\
\end{tabular}\label{3rank}
\end{table}

\begin{remark}
To the best of our knowledge, both \eqref{rank:multi-pod} and \eqref{cong:multi-pod} seem to be new, but the special case of $t=3$ was first found by Chen and Lin \cite{CL2}. Their rank is exactly our multirank for $t=3$, so Theorem~\ref{thm:multi-pod pa} can be viewed as a natural generalization of the work of Chen and Lin as well. Later, Chan and Mao \cite{CM} actually discovered two statistics, rank and crank, both lead to combinatorial explanations comparable to \eqref{rank:multi-pod}. Indeed, their rank is analogous to Bringmann and Lovejoy's rank \cite{BL} for overpartition pairs, while their crank is reminiscent of the third (out of three) rank for overpartition pairs studied by Chen and Lin \cite{CLi}.
\end{remark}

We note that Chan and Mao asked if there is a bijection linking their rank and crank (see the remark after Theorem 1.7 in \cite{CM}). Now this multirank $\overline{r}$ only adds to the mystery of this problem. We illustrate the case with pod pairs of $5$ in Table~\ref{3rank} above, and the interested readers are referred to Chan and Mao's original paper \cite{CM} for the definitions of their rank and crank. Note that all three ranks listed in Table~\ref{3rank} are symmetric, i.e., the number of pod pairs of $n$ with rank $m$ is the same as the number of pod pairs of $n$ with rank $-m$, a fact that easily follows from considering the ``conjugation'' $(\lambda,\mu)\mapsto (\mu,\lambda)$, which reverses the sign of both Chan-Mao crank and multirank, but not for Chan-Mao rank. Interpreting this symmetry combinatorially for Chan-Mao rank might be a starting point of finding a bijection.

We conclude this section with another family of congruences enjoyed by $p_t(n)$, $\overline{p}_t(n)$ and $\hat{p}_t(n)$, which have combinatorial explanations using a modified multirank. For odd prime $t$ and any $t$-colored partition $\overrightarrow{\pi}=(\pi_1,\pi_2,\cdots,\pi_t)\in\mathcal{P}^t$ (resp. $\overline{\mathcal{P}}^t$, $\hat{\mathcal{P}}^t$), we define another multirank $r^{*}(\overrightarrow{\pi})$ as
\begin{align*}
r^{*}(\overrightarrow{\pi})=\sum_{k=1}^{(t-1)/2}k(\ell(\pi_{k})-\ell(\pi_{t-k})).
\end{align*}
And we let $M_t(m,n)$ (resp. $\overline{M}_t(m,n)$, $\hat{M}_t(m,n)$) denote the number of $t$-colored partitions (resp. overpartitions, pods) of $n$ with multirank $r^{*}$ equals $m$. We state the following result, the proof of which follows the same line as Theorem~\ref{thm:multi-over} and Theorem~\ref{thm:multi-pod pa}, without invoking \eqref{phi} or \eqref{psi}, thus is omitted. We note that cases of Theorem~\ref{thm:newmulti} have undoubted appeared in the literature (see \cite{Chen,For,Hir1,New}, to name a few). It is the uniformity in the use of our second multirank $r^*$ to obtain congruences combinatorially that we are trying to emphasize here.
\begin{theorem}\label{thm:newmulti}
Let $t$ be any odd prime, if $n\not \equiv 0 \pmod t$, then for all $0\leq i\leq t-1$, we have
\begin{align*}
M_{t}(i,t,n)&=\dfrac{p_{t}(n)}{t};\\
\overline{M}_{t}(i,t,n)&=\dfrac{\overline{p}_{t}(n)}{t}\equiv 0 \pmod 2;\\
\hat{M}_{t}(i,t,n)&=\dfrac{\hat{p}_{t}(n)}{t}.
\end{align*}
In particular, we have
\begin{align*}
p_t(n)\equiv \hat{p}_t(n)\equiv 0 \pmod t, \text{ and }\; \overline{p}_t(n)\equiv 0 \pmod {2t}.
\end{align*}
\end{theorem}

\section{Cubic partitions, overcubic partitions and the related multi-colored partitions}
\label{sec: cubic}
\subsection{A unified crank analog for cubic partitions and overcubic partitions}
In a series of papers, Eggan \cite{Egg} and H.-C. Chan \cite{Cha1,Cha2,Cha3} independently studied congruence properties of a certain kind of partition $a(n)$, which arise from Ramanujan's cubic continued fraction. This partition function is defined by
\begin{align*}
\sum_{n=0}^{\infty}a(n)q^{n}=\dfrac{1}{(q;q)_{\infty}(q^{2};q^{2})_{\infty}}.
\end{align*}

Partition-theoretically, $a(n)$ can be interpreted as the number of 2-colored partitions of $n$ with colors red and blue subjecting to the restriction that the color blue appears only in even parts. By using identities from the cubic continued fraction found by Ramanujan, Chan derived a result analogous to ``Ramanujan's most beautiful identity'' \cite{Ram}, which implies immediately that
\begin{align}\label{cubic-mod3}
a(3n+2)\equiv0\pmod{3}.
\end{align}

Later, Kim \cite{Kim2} named the partitions enumerated by $a(n)$ as \emph{cubic partitions} and established a crank analog $M'(m,N,n)$ for them such that
\begin{align*}
M'(0,3,3n+2)\equiv M'(1,3,3n+2)\equiv M'(2,3,3n+2)\pmod{3},
\end{align*}
where $M'(m,N,n)$ is the number of cubic partitions of $n$ with Kim's crank congruent to $m$ modulo $N$. Note that the three residue classes divided by Kim's crank are only congruent modulo $3$, not equal with each other. However, Reti \cite{Ret} defined a different crank in his Ph.D. thesis, and his crank could actually break cubic partitions of $3n+2$ into three equinumerous subsets. We will briefly recall his crank here and show that it can be utilized to give similar combinatorial interpretations to overcubic partitions as defined by Kim \cite{Kim4}.

For a given cubic partition $\overrightarrow{\pi}=(\pi_{r},\pi_{b})$, Reti defined a cubic partition crank as
\begin{align}\label{crank:cubic part}
r_c(\pi)=\ell(\pi_{r}^{e})-\ell(\pi_{b}),
\end{align}
where $\pi_{r}^{e}$ is composed of all the even parts in red color. Let $C(m,n)$ be the number of cubic partitions of $n$ with $r_c$ equals $m$, then we have the generating function
\begin{align}\label{gf:cubic-pa}
\sum_{n=0}^{\infty}\sum_{m=-\infty}^{\infty}C(m,n)z^{m}q^{n}=\dfrac{1}{(q;q^{2})_{\infty}(zq^{2};q^{2})_{\infty}(z^{-1}q^{2};q^{2})_{\infty}}.
\end{align}
This sets us up for the next result due to Reti, which clearly implies \eqref{cubic-mod3}.
\begin{theorem}[Theorem~1, p.9\cite{Ret}]
For $n\geq 0$,
\begin{align*}
C(0,3,3n+2)=C(1,3,3n+2)=C(2,3,3n+2)=\dfrac{a(3n+2)}{3}.
\end{align*}
\end{theorem}

\begin{example}
In Table~\ref{Kim-FT} below, we compare Kim's crank with Reti's crank $r_c$ for the $12$ total cubic partitions of $5$. Note that the weighted count of Kim's crank divides $a(5)=12$ into residue classes $\equiv 0,1 \text{ and } 2 \pmod 3$, with size $2$, $5$ and $5$ respectively, while Reti's crank divides $a(5)=12$ into three residue classes, each with equal size $4$.
\end{example}
\begin{table}[htbp]\caption{Cubic partitions of $5$ with two types of cranks}\label{Kim-FT}
\centering
\begin{tabular}{|c||c|c|}
Cubic partitions  & Kim's crank & Reti's crank $r_c$ \\
$\{5_{r},\varnothing\}$ & $(1,5)$ & $0$\\
$\{4_{r}1_{r},\varnothing\}$ & $(1,0)$ & $1$ \\
$\{3_{r}2_{r},\varnothing\}$ & $(1,3)$ & $1$ \\
$\{3_{r}1_{r}1_r,\varnothing\}$ & $(1,-1)$ & $0$  \\
$\{2_{r}2_{r}1_r,\varnothing\}$ & $(1,1)$ & $2$ \\
$\{2_{r}1_{r}1_r1_r,\varnothing\}$ & $(1,-3)$ & $1$\\
$\{1_{r}1_{r}1_r1_r1_r,\varnothing\}$ & $(1,-5)$ & $0$\\
$\{3_{r},2_b\}$ & $(-1,3),(1,4),(1,2)$ & $-1$\\
$\{2_{r}1_r,2_b\}$ & $(-1,0),(1,1),(1,-1)$ & $0$\\
$\{1_{r}1_r1_r,2_b\}$ & $(-1,-3),(1,-2),(1,-4)$ & $-1$\\
$\{1_{r},4_b\}$ & $(-1,2),(1,3),(1,1)$ & $-1$\\
$\{1_{r},2_b2_b\}$ & $(-1,-2),(1,-1),(1,-3)$ & $-2$\\
\end{tabular}
\end{table}

In \cite{Kim4}, Kim introduced the notion of overcubic partitions $\overline{a}(n)$, defined by
\begin{align*}
\sum_{n=0}^{\infty}\overline{a}(n)q^{n}=\dfrac{(-q;q)_{\infty}(-q^{2};q^{2})_{\infty}}{(q;q)_{\infty}(q^{2};q^{2})_{\infty}}.
\end{align*}

Clearly $\overline{a}(n)$ can be interpreted as the overpartition analog of cubic partitions of $n$. By using the theory of modular functions, Kim also derived a result analogous to ``Ramanujan's most beautiful identity'' \cite{Ram} for overcubic partitions, which implies immediately that
\begin{align}
\overline{a}(3n+2)\equiv0\pmod{6}.\label{overcubic part}
\end{align}

Later, Hirschhorn \cite{Hir2} provided an elementary proof of \eqref{overcubic part} via $q$-series dissections. Interestingly, Reti's crank can be generalized to give a combinatorial interpretation for \eqref{overcubic part} as well, just note that now $\ell(\pi)$ counts all parts in $\pi$, both overlined and non-overlined.

Let $\overline{C}(m,n)$ be the number of overcubic partitions of $n$ with the crank $r_c$ equals $m$, then it easily follows that
\begin{align}
\sum_{n=0}^{\infty}\sum_{m=-\infty}^{\infty}\overline{C}(m,n)z^{m}q^{n}=\dfrac{(-q;q^{2})_{\infty}(-zq^{2};q^{2})_{\infty}(-z^{-1}q^{2};q^{2})_{\infty}}{(q;q^{2})_{\infty}
(zq^{2};q^{2})_{\infty}(z^{-1}q^{2};q^{2})_{\infty}}.\label{gf:overcubic part}
\end{align}
Furthermore, we derive the following result that immediately implies \eqref{overcubic part}.
\begin{theorem}\label{thm:overcubic}
For $n\geq0$,
\begin{align*}
\overline{C}(0,3,3n+2)=\overline{C}(1,3,3n+2)=\overline{C}(2,3,3n+2)=\dfrac{\overline{a}(3n+2)}{3}\equiv 0 \pmod 2.
\end{align*}
\end{theorem}
\begin{proof}
First note that the ``toggling involution'' described in the proof of Theorem~\ref{thm:multi-over} still preserves the $r_c$ crank of overcubic partitions, therefore each $\overline{C}(i,3,3n+2), i=0,1,2$ must be an even number. Next, by setting $z=\zeta_{3}=e^{2\pi i/3}$ in \eqref{gf:overcubic part} gives
\begin{align}\label{gf2:overcubic-pa}
 &\sum_{n=0}^{\infty}\sum_{m=-\infty}^{\infty}\overline{C}(m,n)\zeta_{3}^{m}q^{n}=\sum_{n=0}^{\infty}\sum_{i=0}^{2}\overline{C}(i,3,n)\zeta_{3}^{i}q^{n}\\
 =&\dfrac{(-q;q^{2})_{\infty}(-\zeta_{3}q^{2};q^{2})_{\infty}(-\zeta_3^{-1}q^{2};q^{2})_{\infty}}{(q;q^{2})_{\infty}(\zeta_{3}q^{2};q^{2})_{\infty}
 (\zeta_3^{-1}q^{2};q^{2})_{\infty}}\nonumber\\
 =&\dfrac{(-q^{6};q^{6})_{\infty}(-q;q^{2})_{\infty}^{2}(q^{2};q^{2})_{\infty}}{(q^{6};q^{6})_{\infty}}\nonumber\\
 =&\dfrac{(-q^{6};q^{6})_{\infty}}{(q^{6};q^{6})_{\infty}}\sum_{n=-\infty}^{\infty}q^{n^{2}}.\nonumber
\end{align}
Here we need the $\varphi(q)$ version of \eqref{phi} for the final equality. Note that $n^{2}\equiv0~\textrm{or}~1\pmod{3}$ for arbitrary integer $n$, thus the coefficient of $q^{3n+2}$ in \eqref{gf2:overcubic-pa} is always zero. We use the minimal polynomial argument again to complete the proof.
\end{proof}

For subsequent studies either directly involving $a(n)$ or encouraged by the success of $a(n)$, see for example \cite{BO,CC,CLi0,CT,Kim2,Kim3,Sin,Toh,Xio,ZZh,Zhou,ZZ}. In particular, we would like to mention the work of Zhao and Zhong \cite{ZZh} on the so-called ``\emph{cubic partition pairs}'' $b(n)$, whose generating function is given as
\begin{align}\label{gf: 12}
\sum_{n=0}^{\infty}b(n)q^{n}:=\dfrac{1}{(q;q)_{\infty}^{2}(q^{2};q^{2})_{\infty}^{2}}.
\end{align}
As well as the work of H.-C.~Chan and Cooper \cite{CC} on certain $4$-colored partition function $c(n)$, whose generating function is given as
\begin{align}\label{gf: 13}
\sum_{n=0}^{\infty}c(n)q^{n}:=\dfrac{1}{(q;q)_{\infty}^{2}(q^{3};q^{3})_{\infty}^{2}}.
\end{align}

The similar infinite product sides of both \eqref{gf: 12} and \eqref{gf: 13}, together with their common congruence properties modulo $5$ (see \eqref{cong:4c} below), inspired us to consider the problem of exhausting all such $4$-colored partition functions with Ramanujan type congruences modulo $5$. This is fully explained in the next subsection.

\subsection{A unified multirank for $15$ types of $4$-colored partitions}
Following H.-H.~Chan and Toh \cite{CT}, we define the generalized partition function $p_{[c^{l}d^{m}]}(n)$ by
\begin{align*}
\sum_{n=0}^{\infty}p_{[c^{l}d^{m}]}(n)q^{n}:=\dfrac{1}{(q^{c};q^{c})_{\infty}^{l}(q^{d};q^{d})_{\infty}^{m}}.
\end{align*}
For instance, \eqref{gf: 12} is then the case of $(c,d,l,m)=(1,2,2,2)$, while \eqref{gf: 13} is the case of $(c,d,l,m)=(1,3,2,2)$.

We focus on the cases with $l=m=2$. In terms of partition, this reduces to the generating function of certain $4$-colored partitions, say with colors {\bf r}ed, {\bf b}lue, {\bf y}ellow and {\bf o}range, subjecting to the restriction that parts in red and blue are divisible by $c$, while parts in yellow and orange are divisible by $d$. We call these ``\emph{$4$-colored partitions of type $(c,d)$}'', and denote the set of such partitions by $\mathcal{P}_{[c,d]}$. The following multirank enables us to uniformly derive congruences mod $5$ for $p_{[c^2d^2]}(n)$, with all possible values of $c$ and $d$.

\begin{Def}
For $\overrightarrow{\pi}=(\pi_r,\pi_b,\pi_y,\pi_o)\in\mathcal{P}_{[c,d]}$, we define the \emph{4-colored rank} $r_4$ as
\begin{align}\label{4c-rank}
r_4(\pi)=\ell(\pi_{r})-\ell(\pi_{b})+\ell(\pi_{y})-\ell(\pi_{o}).
\end{align}
We use $N_{[c,d]}(m,n)$ to enumerate all $4$-colored partitions of $n$, that are of type $(c,d)$ and with $4$-colored rank $m$.
\end{Def}
\begin{remark}
We note that for the purpose of combinatorially obtaining congruences, there are a number of alternatives that will work equally well as $r_4$. Namely, we can use $s\left(\ell(\pi_{r})-\ell(\pi_{b})\right)+t\left(\ell(\pi_{y})-\ell(\pi_{o})\right)$ for $s,t\not\equiv 0 \pmod 5$, to replace the right hand side of \eqref{4c-rank}. This way, the corresponding generating function will certainly be different, but in the forthcoming proof we will actually evaluate $z$ at a primitive $5$-th root of unity, and the ensuing argument does not depend on which primitive root we have chosen, therefore this change will not prevent us from getting our results. The reason that we have picked $r_4$ as the ``generic'' $4$-color rank is twofold.
\begin{itemize}
\item First it should be acknowledged that in the work of Zhou \cite{Zhou}, Zhou and Zhang \cite{ZZ}, they used exactly the same multirank as $r_4$ to explain congruences mod $5$ for $4$-colored partitions of type $(1,2),(1,3)$ and $(2,3)$ respectively, so our Theorem~\ref{thm:4c} below can be viewed as a natural complement to their work and shows the full power of $r_4$.
\item Secondly, this definition of $r_4$ entitles us to consider a natural rank-preserving bijection between $\mathcal{P}_{[c,d]}$ and $\mathcal{P}_{[d,c]}$: $(\pi_r,\pi_b,\pi_y,\pi_o)\mapsto (\pi_y,\pi_o,\pi_r,\pi_b)$. Consequently, it suffices to consider only the types with $c\leq d$.
\end{itemize}
\end{remark}
Now we are ready to state our main results on the congruences modulo $5$ for $p_{[c^2d^2]}(n)$.
\begin{theorem}\label{thm:4c}
For all integers $1\leq c \leq d, n\geq 0$, suppose $c\equiv i \pmod 5$ and $d\equiv j \pmod 5$, with $0\leq i,j\leq 4$. Then we have
\begin{align}\label{rank:4c}
N_{[c,d]}(k,5,5n+a_{ij})=\frac{p_{[c^2d^2]}(5n+a_{ij})}{5}, \text{ for } 0\leq k \leq 4,
\end{align}
where all possible values of $a_{ij}$ are summarized in Table~\ref{table:aij}. In particular, we have
\begin{align}\label{cong:4c}
p_{[c^2d^2]}(5n+a_{ij})\equiv 0 \pmod 5.
\end{align}
\end{theorem}
\begin{table}[htbp]\caption{Fifteen types of $4$-colored partitions with possible values of $a_{ij}$}\label{table:aij}
\centering
\begin{tabular}{|c|c|}
\hline
$(i,j)$  & $a_{ij}$ \\
\hline
$(0,0)$ & $1,2,3,4$ \\
$(0,1)$ & $2,3,4$ \\
$(0,2)$ & $1,3,4$ \\
$(0,3)$ & $1,2,4$ \\
$(0,4)$ & $1,2,3$ \\
$(1,1)$ & $3,4$ \\
$(1,2)$ & $4$ \\
$(1,3)$ & $2$ \\
$(1,4)$ & $2,3$ \\
$(2,2)$ & $1,3$ \\
$(2,3)$ & $1,4$ \\
$(2,4)$ & $3$ \\
$(3,3)$ & $2,4$ \\
$(3,4)$ & $1$ \\
$(4,4)$ & $1,2$ \\
\hline
\end{tabular}
\end{table}
\begin{proof}
We will first show the proof of type $(1,4)$ to illustrate the main idea, then we will explain its connection with the remaining types and the reason why it will suffice to consider only these fifteen types listed.

We begin with the generating function for $N_{[1,4]}(m,n)$:
\begin{align}\label{gf:14}
\sum_{n=0}^{\infty}\sum_{m=-\infty}^{\infty}N_{[1,4]}(m,n)z^{m}q^{n}=\dfrac{1}{(zq,z^{-1}q;q)_{\infty}(zq^{4},z^{-1}q^{4};q^{4})_{\infty}}.
\end{align}
We also need the following modified version of \eqref{JTP}:
\begin{align}\label{JTP:iden}
\prod_{n=1}^{\infty}(1-q^{n})(1-zq^{n})(1-z^{-1}q^{n})=\sum_{n=0}^{\infty}(-1)^{n}q^{n(n+1)/2}z^{-n}\left(\dfrac{1-z^{2n+1}}{1-z}\right).
\end{align}
Now, setting $z=\zeta_{5}=e^{2\pi i/5}$ in \eqref{gf:14}, we see that
\begin{align}\label{iden1}
 &\sum_{n=0}^{\infty}\sum_{m=-\infty}^{\infty}N_{[1,4]}(m,n)\zeta_{5}^{m}q^{n}=\sum_{n=0}^{\infty}\sum_{k=0}^{4}N_{[1,4]}(k,5,n)\zeta_{5}^{k}q^{n} \\
 =&\dfrac{1}{(\zeta_{5}q,\zeta_{5}^{-1}q;q)_{\infty}(\zeta_{5}q^{4},\zeta_{5}^{-1}q^{4};q^{4})_{\infty}} \nonumber\\
 =&\dfrac{(\zeta_{5}^{2}q,\zeta_{5}^{-2}q,q;q)_{\infty}(\zeta_{5}^2q^{4},\zeta_{5}^{-2}q^{4},q^{4};q^{4})_{\infty}}{(q^{5};q^{5})_{\infty}(q^{20};q^{20})_{\infty}}\nonumber\\
 =&\dfrac{\sum_{m,n=-\infty}^{\infty}(-1)^{m+n}q^{m(m+1)/2+2n(n+1)}\zeta_{5}^{-2m-2n}\left(1-(\zeta_{5}^2)^{2m+1}\right)
 \left(1-(\zeta_{5}^{2})^{2n+1}\right)}
{(1-\zeta_{5}^2)^2(q^{5};q^{5})_{\infty}(q^{20};q^{20})_{\infty}}. \nonumber
\end{align}
Here we use \eqref{JTP:iden} twice in the last equality. Since $m(m+1)/2\equiv0,1,3\pmod{5}$ and $2n(n+1)\equiv0,2,4\pmod{5}$, the power of $q$ is congruent to 2 modulo 5 only for one of the following two cases:
\begin{itemize}
\item[i.] $m(m+1)/2\equiv0\pmod{5}$ and $2n(n+1)\equiv 2\pmod{5}$;
\item[ii.] $m(m+1)/2\equiv3\pmod{5}$ and $2n(n+1)\equiv 4\pmod{5}$.
\end{itemize}
Case i leads to $n\equiv 2\pmod{5}$, while case ii forces $m\equiv 2\pmod5$, both of which result in $(1-(\zeta_5^2)^{2m+1})(1-(\zeta_5^2)^{2n+1})=0$, hence the coefficient of $q^{5n+4}$ in \eqref{iden1} is zero. The case of $5n+3$ can be similarly done. We have established \eqref{rank:4c} for $c=1,d=4$.

Next, note that if $c\equiv c', d\equiv d' \pmod 5$, then $c\binom{m+1}{2}+d\binom{n+1}{2}\equiv c'\binom{m+1}{2}+d'\binom{n+1}{2} \pmod 5$, since $\binom{m+1}{2}$ and $\binom{n+1}{2}$ are both integers. These two expressions are exactly the powers of $q$ as in \eqref{iden1} for type $(1,4)$. Therefore the two types $(c,d)$ and $(c',d')$ should have the same congruences modulo $5$, and it will suffice to consider the fifteen types of $(i,j)$ displayed in Table~\ref{table:aij}. For similar reasons, we see that if one of $i,j$ is $0$, it is actually easier to find the suitable values of $a_{ij}$ that will force $(1-(\zeta_5^2)^{2m+1})(1-(\zeta_5^2)^{2n+1})=0$. And the distinction between the cases with one value of $a_{ij}$ and those with two values of $a_{ij}$, is simply due to the fact that in one case there is a unique residue pair $(m,n)$ which results in both $(1-(\zeta_5^2)^{2m+1})$ and $(1-(\zeta_5^2)^{2n+1})$ being zero, while in the other case there are two possible residue pairs of $(m,n)$, each of which leads to one of $(1-(\zeta_5^2)^{2m+1})$ and $(1-(\zeta_5^2)^{2n+1})$ being zero, just like the type $(1,4)$ we have discussed above.
\end{proof}
\begin{remark}
Except for types $(1,1),(1,2),(1,3),(2,3)$, the congruences in \eqref{cong:4c} seem to be new, although all of them are of the same nature, a point that hopefully has been illustrated by \eqref{rank:4c} and its unified proof we supplied above. We make two more remarks concerning two specific types. First note that Kim \cite{Kim3} also gave a partition statistic which explained \eqref{cong:4c} for type $(1,2)$, i.e., cubic partition pairs. Unfortunately, the residue classes divided by his statistic were again merely congruent modulo $5$, not equinumerous.
Moreover, note that type $(1,1)$ actually coincides with the $t=5$ case of Theorem \ref{Gmulti} (i).
\end{remark}

\section{Final Remarks}
In this paper, we define a unified multirank $\overline{r}$ for multipartitions, multi-overpartitions and multi-pods, a unified multirank $r_4$ for fifteen types of $4$-colored partitions. And we generalize Reti's crank analog $r_c$ to overcubic partitions. These statistics are then shown to give combinatorial interpretations to a series of partition congruences. The three classical theta functions and properties of the $p$-th root of unity have always been the key ingredients in these derivations. This rank/crank-oriented approach is notably different from most of the studies in the literature.

We conclude with some suggestions for further investigation. Firstly, it is reasonable to envision generalizing Theorem~\ref{thm:4c} to odd prime moduli other than $5$. The most natural generalization of \eqref{cong:4c} that comes to our mind is, for odd prime $t\geq 7$,
\begin{align*}
p_{[i^{t-3}j^{t-3}]}(tn+a_{ij})\equiv 0 \pmod t,
\end{align*}
which is reminiscent of the second family of congruences for multipartitions in Theorem~\ref{Gmulti}. Following the same line of proving Theorem~\ref{thm:4c}, such analogous results can be derived case by case. But an explicit formula for computing $a_{ij}$ is more desirable and is necessary for full generalization.

Secondly, just as Ramanujan's original congruences for $p(n)$ have versions modulo high powers of prime, many of the partition functions studied in this paper do enjoy congruences modulo high powers of prime. But it seems few of these high power moduli congruences have been explained combinatorially using appropriate statistics. New ideas need to emerge for discovering these statistics.

Lastly, since moments of rank/crank have been shown to be quite important in the study of further congruences, we plan to investigate along this line and have already derived some results on the weighted count of the associated rank/crank generating functions of various partitions discussed here. These results will appear in \cite{FT} .

\section*{Acknowledgement}
The authors would like to thank Frank Garvan for sharing with us Zoltan Reti's Ph.D. thesis, and for his valuable suggestions that have improved this paper to a great extent. The authors also thank acknowledge the helpful suggestions made by the referee.

Both authors were supported by the Fundamental Research Funds for the Central Universities (No.~CQDXWL-2014-Z004) and the National Science Foundation of China (No.~115010\\61).

\end{document}